\documentclass{article}
\usepackage{amsfonts}
\usepackage{amssymb}
\usepackage{amsmath}
\usepackage[center]{caption2}
\usepackage[verbose]{wrapfig}
\usepackage{geometry}

\newtheorem{theorem}{Theorem}

\newtheorem{definition}[theorem]{Definition}

\newtheorem{lemma}[theorem]{Lemma}
\newtheorem{notation}[theorem]{Notation}

\newtheorem{proposition}[theorem]{Proposition}
\newtheorem{remark}[theorem]{Remark}

\newenvironment{proof}[1][Proof]{\noindent\textbf{#1.} }{\ \rule{0.5em}{0.5em}}
\input{tcilatex}
\begin{document}

\title{Normality of Monomial Ideals}
\author{Ibrahim Al-Ayyoub}
\maketitle

\begin{abstract}
Given the monomial ideal $I=(x_{1}^{\alpha _{1}},\ldots ,x_{n}^{\alpha
_{n}})\subset K[x_{1},\ldots ,x_{n}]$ where $\alpha _{i}$ are positive
integers and $K$ a field and let $J$ be the integral closure of $I$ . It is
a challenging problem to translate the question of the normality of $J$ into
a question about the exponent set $\Gamma (J)$\ and the Newton polyhedron $%
NP(J)$. A relaxed version of this problem is to give necessary or sufficient
conditions on $\alpha _{1},\ldots ,\alpha _{n}$ for the normality of $J$. We
show that if $\alpha _{i}\in \{s,l\}$ with $s$ and $l$ arbitrary positive
integers, then $J$ is normal.
\end{abstract}

\section*{Introduction \ \ \ }

Let $I$ be an ideal in a Noetherian ring $R$. The integral closure of $I$ is
the ideal $\overline{I}$ that consists of all elements of $R$ that satisfy
an equation of the form 
\begin{equation*}
x^{n}+a_{1}x^{n-1}+\cdots +a_{n-1}x+a_{n}=0,\ \ \ \ \ a_{i}\in I^{i}
\end{equation*}%
The ideal $I$ is said to be integrally closed if $I=\overline{I}$. Clearly
one has that $I\subseteq \overline{I}\subseteq \sqrt{I}$. An ideal is called
normal if all of its positive powers are integrally closed. It is known that
if $R$ is a normal integral domain, then the Rees algebra $R[It]=\oplus
_{n\in N}I^{n}t^{n}$ is normal if and only if $I$ is a normal ideal of $R$ .
This brings up the importance of normality of ideals as the Rees algebra is
the algebraic counterpart of blowing up a scheme along a closed subscheme.

\ \ \ 

It is well known that the integral closure of monomial ideal in a polynomial
ring is again a monomial ideal, see \cite{Swanson-Huneke} or \cite{Vitulli}
for a proof. The problem of finding the integral closure for a monomial
ideal $I$ reduces to finding monomials $r$, integer $i$ and monomials $%
m_{1},m_{2},\ldots ,m_{i}$ in $I$ such that $r^{i}+m_{1}m_{2}\cdots m_{i}=0$%
, see \cite{Swanson-Huneke}. Geometrically, finding the integral closure of
monomial ideals $I$ in $R=K[x_{0},\ldots ,x_{n}]$ is the same as finding all
the integer lattice points in the convex hull $NP(I)$\ (the Newton
polyhedron of $I$) in $\mathbb{R}^{n}$ of $\Gamma (I)$ (the Newton polytope
of $I$) where $\Gamma (I)$ is the set of all exponent vectors of all the
monomials in $I$. This makes computing the integral closure of monomial
ideals simpler.

\ \ \ \ \ 

A power of an integrally closed monomial ideal need not be integrally
closed. For example, let $J$ be the integral closure of $%
I=(x^{4},y^{5},z^{7})\subset K[x,y,z]$. Then $J^{2}$ is not integrally
closed (observe that $y^{3}z^{3}\in J$ as $\left( y^{3}z^{3}\right)
^{5}=y^{5}y^{5}y^{5}z^{7}z^{8}\in I^{5}$. Now $x^{2}y^{4}z^{5}\in \overline{%
J^{2}}$ since $\left( x^{2}y^{4}z^{5}\right) ^{2}=\left( x^{4}\cdot
y^{5}\right) \left( y^{3}z^{3}\cdot z^{7}\right) \in \left( J^{2}\right)
^{2} $. On the other hand we used the algebra software Singular\ \cite%
{Singular} to show that $x^{2}y^{4}z^{5}\notin J^{2}$). However, a nice
result of Reid et al. \cite[Proposition 3.1]{Reid} states that if the first $%
n-1$ powers of a monomial ideal, in a polynomial ring of $n$ variables over
a field, are integrally closed, then the ideal is normal. For the case $n=2$
this follows from the celebrated theorem of Zariski \cite{ZS}\ that asserts
that the product of integrally closed ideals in a 2-dimensional regular ring
is again integrally closed.

\ \ \ \ \ \ \ 

In general, there is no good characterization for normal monomial ideals. It
is a challenging problem to translate the question of normality of a
monomial ideal $I$ into a question about the exponent set $\Gamma (I)$\ and
the Newton polyhedron $NP(I)$. Under certain hypotheses, some necessary
conditions are given. Faridi \cite{Faridi} gives necessary conditions on the
degree of the generators of a normal ideal in a graded domain. Vitulli \cite%
{Vitulli} investigated the normality for special monomial ideals in a
polynomial ring over a field.

\ \ \ 

For $\mathbf{\alpha =}(\alpha _{1},\ldots ,\alpha _{n})\in \mathbb{N}^{n}$
let $I(\mathbf{\alpha })$ be the integral closure of $(x_{1}^{\alpha
_{1}},\ldots ,x_{n}^{\alpha _{n}})\subset K[x_{1},\ldots ,x_{n}]$. Reid et.
al. \cite{Reid} showed that if $\mathbf{\alpha =(}\alpha _{1},\ldots ,\alpha
_{n})$\ with pairwise relatively prime entries, then the ideal $I(\mathbf{%
\alpha })$ is normal if and only if the additive submonoid $\Lambda
=\left\langle 1/\alpha _{1},\ldots ,1/\alpha _{n}\right\rangle $ of $\mathbb{%
Q}_{\geq }$ is quasinormal, that is, whenever $x\in \Lambda $ and $x\geq p$
for some $p\in \mathbb{N}$, there exist rational numbers $y_{1},\ldots
,y_{p} $ in $\Lambda $ with $y_{i}\geq 1$ for all $i$ such that $%
x=y_{1}+\cdots +y_{p}$. Thus for the case where $\alpha _{1},\ldots ,\alpha
_{n}$\ are pairwise relatively prime, the normality condition on the $n$%
-dimensional monoid is reduced to the quasinormality condition on the $1$%
-dimensional monoid. Another nice result of Reid et. al. \cite{Reid} is that
the monomial ideal $I(\mathbf{\alpha })$ is normal if $\gcd (\alpha
_{1},\ldots ,\alpha _{n})>n-2$. In particular, if $n=3$ and $\gcd (\alpha
_{1},\alpha _{2},\alpha _{3})\neq 1$, then $I(\mathbf{\alpha })$ is normal.
Therefore, in $k[x_{1},x_{2},x_{3}]$ it remains to investigate the normality
of $I(\mathbf{\alpha })$ whenever $\gcd (\alpha _{1},\alpha _{2},\alpha
_{3})=1$ and the integers are not pairwise relatively prime.

\ \ 

A important result of Reid et. al. \cite{Reid}, which we use to improve our
result in this paper, is as following. Choose $i$ and set $c=\func{lcm}%
(\alpha _{1},\ldots ,\widehat{\alpha _{i}},\ldots ,\alpha _{n})$. Put $%
\mathbf{\alpha }^{\prime }=\mathbf{(}\alpha _{1},\ldots ,\alpha
_{i-1},\alpha _{i}+c,\alpha _{i+1},\ldots ,\alpha _{n})$. If $I(\mathbf{%
\alpha }^{\prime })$ is normal then $I(\mathbf{\alpha })$ is normal.
Conversely, If $I(\mathbf{\alpha })$ is normal and $\alpha _{i}\geq c$, then 
$I(\mathbf{\alpha }^{\prime })$ is normal.

\ \ \ \ 

The goal of \ this paper is to show that the integral closure of the ideal $%
(x_{1}^{\alpha _{1}},\ldots ,x_{n}^{\alpha _{n}})\subset K[x_{1},\ldots
,x_{n}]$\ is normal provided that $\alpha _{i}\in \{s,l\}$ with $s$ and $l$
arbitrary positive integers. The\ following theorem provide us with a
technique that we mainly depend on to prove the integral closedness.

\begin{theorem}
\label{MainIC}(Proposition 15.4.1, \cite{Swanson-Huneke}) Let $I$ be a
monomial ideal in the polynomial ring $R=K[x_{1},\ldots ,x_{n}]$ with $K$ a
field. If $I$\ is primary to $(x_{1},\ldots ,x_{n})$ and $\overline{I}\cap
(I:(x_{1},\ldots ,x_{n}))\subseteq I,$ \ then $I$\ is integrally closed.
\end{theorem}

\begin{proposition}
\label{IntegContain}(Corollary 5.3.2, \cite{Swanson-Huneke}) If $I\subseteq
J $ are ideals in a ring $R$, then $J\subseteq \overline{I}$ if and only if
each element in some generating set of $J$ is integral over $I$.
\end{proposition}

\section*{Certain Normal Monomial Ideals}

Let $(x_{1}^{s},\ldots ,x_{m}^{s},y_{1}^{l},\ldots ,y_{n}^{l})\subset
K[x_{1},\ldots ,x_{m},y_{1},\ldots ,y_{n}]$ with $K$ a field, $x_{i}$ and $%
y_{i}$ indeterminates over $K$, and $s$ and $l$ positive integers such that
(without loss of generality) $l\geq s$.

\begin{notation}
For the remaining of this paper fix positive integers $s$ and $l$ with $%
l\geq s$ and let $\lambda _{a}=\left\lceil a\dfrac{l}{s}\right\rceil $ where 
$a$ is any integer. Also, let $k$ be any positive integer.
\end{notation}

Let $x$ and $y$ be positive integers and write $x=ts+r$ with $1\leq r\leq s$%
. Then $y\left\lceil \dfrac{x}{s}\right\rceil =y\dfrac{x+s-r}{s}=y\dfrac{s-r%
}{s}+y\dfrac{x}{s}\leq y\dfrac{s-r}{s}+\left\lceil y\dfrac{x}{s}\right\rceil 
$. Therefore, $\left\lceil y\dfrac{x}{s}\right\rceil \geq y\left(
\left\lceil \dfrac{x}{s}\right\rceil -\dfrac{s-r}{s}\right) $. This
inequality helps to prove the following lemma which is a key in this paper.

\begin{lemma}
\label{lemda-inequality}If\ $i\in \{0,1,\ldots ,ks\}$, then $%
kl(ks-i-1)+\lambda _{i}\geq (ks-i)(\lambda _{ks-1}-\frac{s-r}{s})$, where $%
(ks-1)l=ts+r$ with $1\leq r\leq s.$
\end{lemma}

\begin{proof}
By the note before the lemma we have\ $kl(ks-i-1)+\lambda _{i}=\left\lceil 
\frac{[ks(ks-i-1)+i]l}{s}\right\rceil =\left\lceil (ks-i)\frac{(ks-1)l}{s}%
\right\rceil \geq (ks-i)\left( \left\lceil \frac{(ks-1)l}{s}\right\rceil -%
\frac{s-r}{s}\right) =(ks-i)\left( \lambda _{ks-1}-\frac{s-r}{s}\right) $.
\end{proof}

\begin{definition}
\label{J_k&F_k}Let $F_{k}=\{x_{i_{1}}\cdots x_{i_{ks-a}}y_{j_{1}}\cdots
y_{j_{\lambda _{a}}}\mid a=0,1,2,\ldots ,ks$, $1\leq i_{1}\leq i_{2}\leq
\cdots \leq i_{ks-a}\leq m$, and $1\leq j_{1}\leq j_{2}\leq \cdots \leq
j_{\lambda _{a}}\leq n\}$, $J_{k}$ the ideal generated by all the monomials
in $F_{k}$, and $I_{k}=(x_{1}^{ks},\ldots ,x_{m}^{ks},y_{1}^{kl},\ldots
,y_{n}^{kl})\subset K[x_{1},\ldots ,x_{m},y_{1},\ldots ,y_{n}]$. Also, let $%
J=J_{1}$, $F=F_{1}$, and $I=I_{1}$.
\end{definition}

\begin{lemma}
\label{IntegralMono} $J_{k}$ is integral over the ideal $I_{k}$, that is, $%
J_{k}\subseteq \overline{I_{k}}$.
\end{lemma}

\begin{proof}
By Proposition \ref{IntegContain} it suffices to show that every element of $%
F_{k}$ is integral over $I_{k}$. Note $x_{i_{1}}^{ksl}\cdots
x_{i_{ks-a}}^{ksl}$ $\in I_{k}^{l(ks-a)}$ and $y_{j_{1}}^{ksl}\cdots
y_{j_{\lambda _{a}}}^{ksl}\in I_{k}^{s\lambda _{a}}$. Also note $%
l(ks-a)+s\lambda _{a}=ksl-la+s$ $\left\lceil a\dfrac{l}{s}\right\rceil \geq
ksl.$\ Therefore, $\left( x_{i_{1}}\cdots x_{i_{ks-a}}y_{j_{1}}\cdots
y_{j_{\lambda _{a}}}\right) ^{ksl}\in I_{k}^{ksl}$.
\end{proof}

\ 

The figure below is an illustration of $J_{3}$ $\subset $ $K[x,y,z]$ with $%
s=2$, $l=7$ and $I=(x^{s},y^{s},z^{l})$. In this case $%
I_{3}=(x^{3s},y^{3s},z^{3l})=(x^{6},y^{6},z^{21})$ and $F_{3}=%
\{x^{i}y^{j}z^{\lambda _{6-(i+j)}}\mid i+j=0,1,2,3,4,5,6$ \textit{and} $%
\lambda _{a}=\left\lceil \frac{7a}{2}\right\rceil $ $\}$. The elements of $%
F_{3}$ are represented by black circles. From the figure it is clear that
the set $F_{3}$ minimally generates $\overline{I_{3}}$.%
\begin{equation*}
\FRAME{itbpF}{5.8193in}{3.4212in}{0in}{}{}{Figure}{\special{language
"Scientific Word";type "GRAPHIC";maintain-aspect-ratio TRUE;display
"USEDEF";valid_file "T";width 5.8193in;height 3.4212in;depth
0in;original-width 5.7605in;original-height 3.3754in;cropleft "0";croptop
"1";cropright "1";cropbottom "0";tempfilename
'L872DP00.wmf';tempfile-properties "XPR";}}
\end{equation*}%
\ \ \ \ \ \ \ \ \ \ \ \ \ 

Later we will prove that $J_{k}$ is the integral closure of $I_{k}.$

\begin{lemma}
\label{J^k=J_k}$J^{k}=J_{k}$.
\end{lemma}

\begin{proof}
We show $J_{k}J=J_{k+1}$. Let $x_{i_{1}}\cdots x_{i_{s-a}}y_{j_{1}}\cdots
y_{j_{\lambda _{a}}}\in F$ and $x_{i_{1}}\cdots x_{i_{ks-b}}y_{j_{1}}\cdots
y_{j_{\lambda _{b}}}\in F_{k}$. Multiplying these two monomials we get $%
x_{h_{1}}\cdots x_{h_{(k+1)s-(b+a)}}y_{t_{1}}\cdots y_{t_{\lambda
_{a}+\lambda _{b}}}$ (with $1\leq h_{1}\leq h_{2}\leq \ldots \leq m$ and $%
1\leq t_{1}\leq t_{2}\leq \ldots \leq n$). This is a multiple of $%
x_{h_{1}}\cdots x_{h_{(k+1)s-(b+a)}}y_{t_{1}}\cdots y_{t_{\lambda
_{a+b}}}\in J_{k+1}$ as $\lambda _{a+b}\leq \lambda _{a}+\lambda _{b}$. To
show the other inclusion let $x_{i_{1}}\cdots
x_{i_{(k+1)s-a}}y_{j_{1}}\cdots y_{j_{\lambda _{a}}}\in F_{k+1}$. If $a\geq
ks$, write $a=ks+r$ with $0\leq r\leq s$, then $\lambda _{a}=\lambda
_{ks+r}=\left\lceil (ks+r)\dfrac{l}{s}\right\rceil =kl+\lambda _{r}$. Thus
this monomial equals $x_{i_{1}}\cdots x_{i_{s-r}}y_{j_{1}}\cdots
y_{j_{\lambda _{r}+kl}}$. But $y_{j_{1}}\cdots y_{j_{kl}}\in F_{k}$ and $%
x_{i_{1}}\cdots x_{i_{s-r}}y_{j_{kl+1}}\cdots y_{j_{kl+\lambda _{r}}}\in F$
as $0\leq s-r\leq s$. If $a<ks$, then $x_{i_{1}}\cdots
x_{i_{(k+1)s-a}}y_{j_{1}}\cdots y_{j_{\lambda _{a}}}=x_{t_{1}}\cdots
x_{t_{s}}x_{h_{1}}\cdots x_{h_{ks-a}}y_{j_{1}}\cdots y_{j_{\lambda _{a}}}\in
JJ_{k}$ as $x_{t_{1}}\cdots x_{t_{s}}\in J$ and $x_{h_{1}}\cdots
x_{h_{ks-a}}y_{j_{1}}\cdots y_{j_{\lambda _{a}}}\in J_{k}$.
\end{proof}

\ \ \ \ \ \ \ \ 

The main goal of this paper is to prove the following theorem

\begin{theorem}
\label{MainThm}The integral closure of the ideal $(x_{1}^{\alpha
_{1}},\ldots ,x_{n}^{\alpha _{n}})\subset K[x_{1},\ldots ,x_{n}]$\ is
normal, where $\alpha _{i}\in \{s,l\}$ with $s$ and $l$ arbitrary positive
integers. Or equivalently, the ideal $J$ is normal.
\end{theorem}

By Lemma \ref{IntegralMono} and since $I_{k}\subseteq J_{k}$ we have 
\begin{equation*}
I_{k}\subseteq J_{k}\subseteq \overline{I_{k}}\subseteq \overline{J_{k}}
\end{equation*}%
We will use Theorem \ref{MainIC} to show that $J_{k}$\ is integrally closed,
hence $J_{k}$ is the integral closure of $I_{k}$. Therefore we need the
following.

\begin{remark}
\label{Qideal}Let $R=K[x_{1},\ldots ,x_{m},y_{1},\ldots ,y_{n}]$. For $1\leq
i\leq m$, it is easy to see that $(J_{k}:(x_{i}))/J_{k}$ is generated by $%
\{z_{i_{1}}\cdots z_{i_{ks-a-1}}w_{j_{1}}\cdots w_{j_{_{\lambda _{a}}}}\mid
a=0,\ldots ,ks-1$; $1\leq i_{1}\leq i_{2}\leq \cdots \leq i_{ks-a-1}\leq m$
and $1\leq j_{1}\leq j_{2}\leq \cdots \leq j_{\lambda _{a}}\leq n\}$ where $%
z_{i}$ and $w_{i}$ are the images of $x_{i}$ and $y_{i}$, respectively, in $%
R/J_{k}$. Also, for $1\leq j\leq n$ note that $(J_{k}:(y_{j}))/J_{k}$ is
generated by $\{z_{i_{1}}\cdots z_{i_{ks-b}}w_{_{j_{1}}}\cdots
w_{_{j_{_{\lambda _{b}}}-1}}\mid b=1,\ldots ,ks$; $1\leq i_{1}\leq i_{2}\leq
\cdots \leq i_{ks-b}\leq m$ and $1\leq j_{1}\leq j_{2}\leq \cdots \leq
j_{\lambda _{b}}\leq n\}$. As the intersection of two monomial ideals is
generated by the set of the least common multiples of the generators of the
two ideals, it follows that $(J_{k}:(x_{1},\ldots ,x_{m},y_{1},\ldots
,y_{n}))/J_{k}$ is generated by$\ \{z_{i_{1}}\cdots
z_{i_{ks-e}}w_{_{j_{1}}}\cdots w_{_{j_{_{\lambda _{e}}}-1}}\mid e=1,\ldots
,ks$; $1\leq i_{1}\leq i_{2}\leq \cdots \leq i_{ks-e}\leq m$ and $1\leq
j_{1}\leq j_{2}\leq \cdots \leq j_{\lambda _{e}}\leq n\}$.
\end{remark}

\begin{lemma}
\label{MainLemma}The ideal $J_{k}$ is integrally closed.
\end{lemma}

\begin{proof}
By Theorem \ref{MainIC} we need to show that none of the preimages, in $%
K[x_{1},\ldots ,x_{m},y_{1},\ldots ,y_{n}]$, of the monomial generators of $%
(J_{k}:(x_{1},\ldots ,x_{m},y_{1},\ldots ,y_{n}))/J_{k}$ is in $\overline{%
J_{k}}$. Assume not, that is, assume $\sigma =x_{i_{1}}\cdots
x_{i_{ks-e}}y_{_{j_{1}}}\cdots y_{_{j_{_{\lambda _{e}}}-1}}\in \overline{%
J_{k}}$ for some $e\in \{1,\ldots ,ks\}$. This implies $\sigma ^{d}\in
J_{k}^{d}$ for some positive integer $d$, thus $\sigma ^{d}=$ $%
x_{i_{1}}^{d}x_{i_{2}}^{d}\cdots x_{i_{ks-e}}^{d}y_{_{j_{1}}}^{d}\cdots
y_{_{j_{_{\lambda _{e}}}-1}}^{d\ }$equals the following product of products
of the generators of $J_{k}$

\ \ \ \ \ \ 

$\beta \ \tprod\limits_{1\leq j_{1}\leq \cdots \leq j_{kl}\leq n}\left(
y_{j_{1}}y_{j_{2}}\cdots y_{j_{kl}}\right) ^{c_{j_{1},\ldots ,j_{kl}}}$

$\ \tprod\limits_{\substack{ 1\leq i_{1}\leq m  \\ 1\leq j_{1}\leq \cdots
\leq j_{\lambda _{ks-1}}\leq n}}(x_{i_{1}}y_{j_{1}}y_{j_{2}}\cdots
y_{j_{\lambda _{ks-1}}})^{l_{i_{1},j_{1},\ldots ,j_{\lambda _{ks-1}}}}$

$\ \prod\limits_{\substack{ 1\leq i_{1}\leq i_{2}\leq m  \\ 1\leq j_{1}\leq
\cdots \leq j_{\lambda _{ks-2}}\leq n}}(x_{i_{1}}x_{i_{2}}y_{j_{1}}y_{j_{2}}%
\cdots y_{j_{\lambda _{ks-2}}})^{l_{i_{1},i_{2},j_{1},\ldots ,j_{\lambda
_{ks-2}}}}$

$\ \ \ \ \ \ \ \ \vdots $

$\prod\limits_{\substack{ 1\leq 1i_{1}\leq i_{2}\leq \cdots \leq
i_{ks-2}\leq m  \\ 1\leq j_{1}\leq \cdots \leq j_{\lambda _{2}}\leq n}}%
(x_{i_{1}}\cdots x_{i_{ks-2}}y_{j_{1}}\cdots y_{j_{\lambda
_{2}}})^{l_{i_{1},i_{2},\ldots ,i_{ks-2},j_{1},\ldots ,j_{\lambda _{2}}}}$

$\prod\limits_{\substack{ 1\leq i_{1}\leq i_{2}\leq \cdots \leq i_{ks-1}\leq
m  \\ 1\leq j_{1}\leq \cdots \leq j_{\lambda _{1}}\leq n}}(x_{i_{1}}\cdots
x_{i_{ks-1}}y_{j_{1}}\cdots y_{j_{\lambda _{1}}})^{l_{i_{1},i_{2},\ldots
,i_{ks-1},j_{1},\ldots ,j_{\lambda _{1}}}}$

$\ \prod\limits_{1\leq i_{1}\leq i_{2}\leq \cdots \leq i_{ks}\leq m}\left(
x_{i_{1}}x_{i_{2}}\cdots x_{i_{ks}}\right) ^{l_{i_{1}},_{i_{2}},\ldots
,_{i_{ks}}}$\newline
\ \ \ \ \ \newline
where $\beta \ $is some monomial, $c_{j_{1},\ldots ,j_{kl}}$ and$\
l_{i_{1},\ldots ,i_{t},j_{1},\ldots ,j_{\lambda _{ks-t}}}$ (with $1\leq
t\leq ks$) are nonnegative integers. For $1\leq t\leq ks$ let $%
L_{t}=\sum\limits_{\substack{ 1\leq i_{1}\leq i_{2}\leq \cdots \leq
i_{t}\leq m  \\ 1\leq j_{1}\leq \cdots \leq j_{\lambda _{ks-t}}\leq n}}%
l_{i_{1},\ldots ,i_{t},j_{1},\ldots ,j_{\lambda _{ks-t}}}$ and let $%
C=\sum\limits_{1\leq j_{1}\leq \cdots \leq j_{kl}\leq n}c_{j_{1},\ldots
,j_{kl}}$. By summing powers we have%
\begin{equation}
L_{ks}+L_{ks-1}+\cdots +L_{3}+L_{2}+L_{1}+C=d  \label{(1)}
\end{equation}%
Also, by the total-degree count of the monomial $x_{i_{1}}\cdots
x_{i_{ks-e}} $ we have the following equality\newline
\begin{equation}
(ks)L_{ks}+(ks-1)L_{ks-1}+\cdots +3L_{3}+2L_{2}+L_{1}+\varepsilon =(ks-e)d
\label{(2)}
\end{equation}%
\newline
where $\varepsilon $ is the total-degree of the monomial $x_{i_{1}}\cdots
x_{i_{ks-e}}$ in $\beta $. By the total-degree count of the monomial $%
y_{1}\cdots y_{j_{_{\lambda _{e}}}-1}$ we must have the following inequality%
\newline
\begin{equation}
\lambda _{1}L_{ks-1}+\lambda _{2}L_{ks-2}+\cdots +\lambda
_{ks-3}L_{3}+\lambda _{ks-2}L_{2}+\lambda _{ks-1}L_{1}+Ckl\leq (\lambda
_{e}-1)d  \label{(3)}
\end{equation}%
\ \newline
We finish the proof by showing that $\left( \ref{(1)}\right) $, $\left( \ref%
{(2)}\right) $, and $\left( \ref{(3)}\right) $ can not hold simultaneously.

\ \ \ \ 

From $\left( \ref{(1)}\right) $ and $\left( \ref{(2)}\right) $ 
\begin{equation}
C=(ks-1)L_{ks}+(ks-2)L_{ks-1}+\cdots +2L_{3}+L_{2}+\varepsilon -(ks-e-1)d
\label{(4)}
\end{equation}%
\newline
Recall, $(ks-1)l=ts+r$ with $1\leq r\leq s$ and $\lambda _{ks-1}<\lambda
_{ks}=kl$. Now consider the left-hand side of (\ref{(3)}) 
\begin{eqnarray*}
&&\lambda _{1}L_{ks-1}+\lambda _{2}L_{ks-2}+\cdots +\lambda
_{ks-3}L_{3}+\lambda _{ks-2}L_{2}+\lambda _{ks-1}L_{1}+Ckl \\
&=&\left[ \sum\limits_{i=0}^{ks-1}[kl(ks-1-i)+\lambda _{i}]L_{ks-i}\right]
+\varepsilon kl-kl(ks-e-1)d\text{ \ (By }\left( \text{\ref{(4)}}\right) 
\text{ )} \\
&\geq &\left[ \sum\limits_{i=0}^{ks-1}(ks-i)(\lambda _{ks-1}-\frac{s-r}{s}%
)L_{ks-i}\right] +\varepsilon kl-kl(ks-e-1)d\text{ \ (by Lemma \ref%
{lemda-inequality}\newline
)} \\
&\geq &(\lambda _{ks-1}-\frac{s-r}{s})(ks-e)d-kl(ks-e-1)d\text{ \ \ \ \ \ \
\ \ \ \ \ \ \ \ \ \ \ \ \ \ \ \ (By }\left( \text{\ref{(2)}}\right) \text{ )}
\\
&=&\frac{(ks-1)l}{s}(ks-e)d-kl(ks-e-1)d \\
&=&\left( \dfrac{e}{s}l\right) d \\
&>&(\lambda _{e}-1)d\text{.}
\end{eqnarray*}%
This is a contradiction to $\left( \ref{(3)}\right) $ as required.
\end{proof}

\ \ \ \ \ 

\begin{proof}
\textbf{(of Theorem \ref{MainThm})} The proof follows by the above lemma and
Lemma \ref{J^k=J_k}.
\end{proof}

\ \ \ \ \ 

We have already proved that if $\mathbf{\alpha =}(\alpha _{1},\ldots ,\alpha
_{n})\in \mathbb{N}^{n}$ with the entries of $\mathbf{\alpha }$ consisting
of two positive integers, then $I(\mathbf{\alpha })$, the integral closure
of $(x_{1}^{\alpha _{1}},\ldots ,x_{n}^{\alpha _{n}})\subset K[x_{1},\ldots
,x_{n}]$, is normal. Noting that the ideal $I(x^{4},y^{5},z^{7})\subset
K[x,y,z]$ is not normal, the following question arises: when is $I(\mathbf{%
\alpha })$ normal provided that $\mathbf{\alpha }$ consists of three
distinct positive integers? In the proposition below we give a partial
answer for this question.

\begin{theorem}
(Theorem 5.1, \cite{Reid}) Let $\mathbf{\alpha =}(\alpha _{1},\ldots ,\alpha
_{n})\in \mathbb{N}^{n}$, $c=\func{lcm}(\alpha _{1},\ldots ,\alpha _{n-1})$.
Let $I(\mathbf{\alpha })$ be the integral closure of $(x_{1}^{\alpha
_{1}},\ldots ,x_{n}^{\alpha _{n}})\subset K[x_{1},\ldots ,x_{n}]$ and $I(%
\mathbf{\alpha }^{\prime })$ the integral closure of $(x_{1}^{\alpha
_{1}},\ldots ,x_{n-1}^{\alpha _{n-1}},x_{n}^{\alpha _{n}+c})\subset
K[x_{1},\ldots ,x_{n}]$. If $I(\mathbf{\alpha }^{\prime })$ is normal, then $%
I(\mathbf{\alpha })$ is normal. Conversely, If $I(\mathbf{\alpha })$ is
normal and $\alpha _{n}\geq c$, then $I(\mathbf{\alpha }^{\prime })$ is
normal.
\end{theorem}

\begin{proposition}
If $\mathbf{\alpha =}(\alpha _{1},\ldots ,\alpha _{n})\in \mathbb{N}^{n}$
with $\alpha _{i}\in \{s,l\}$ for $i=1,\ldots ,n-1$ such that $s$ divides $l$
and $l$ divides $\alpha _{n}$, then $I(\mathbf{\alpha })$ is normal.
\end{proposition}

\begin{proof}
We proceed by induction on the integer $q=\alpha _{n}/l$. By Theorem\textbf{%
\ \ref{MainThm}} the ideal $I(\mathbf{\alpha })$ is normal whenever $q=1$.
Note $l=\func{lcm}\{s,l\}$ as $s$ divides $l$. Assume $I(\mathbf{\alpha })$
is normal for $\mathbf{\alpha =}(\alpha _{1},\ldots ,\alpha _{n-1},ql)$ with 
$\alpha _{i}\in \{s,l\}$ for $i=1,\ldots ,n-1$. Then by the above Theorem $I(%
\mathbf{\alpha }^{\prime })$ is normal where $\mathbf{\alpha }^{\prime }%
\mathbf{=}(\alpha _{1},\ldots ,\alpha _{n-1},ql+l)$.
\end{proof}

\ \ \ \

{\small \ \ \ \ \ \ }

{\small Ibrahim Al-Ayyoub, assistant professor.}

{\small Department of Mathematics and Statistics.}

{\small Jordan University of Science and Technology.}

{\small P O Box 3030, Irbid 22110, Jordan.}

{\small Email address: iayyoub@just.edu.jo}

\end{document}